\documentclass[a4paper,reqno]{amsart}
\usepackage{newcent} 
\usepackage[T1]{fontenc}
\usepackage[centertags]{amsmath}
\usepackage{amsfonts}
\usepackage{amssymb}
\usepackage{amsthm}
\usepackage{newlfont}
\usepackage{fancyhdr,fancyvrb}
\usepackage{graphicx}
\usepackage{pinlabel}
\usepackage{nextpage}
\usepackage{color}
\usepackage{hyperref}
\usepackage{esint}
\usepackage{soul}
\usepackage{Mat-Mor-San-13-private}

\numberwithin{equation}{section}

\newtheorem{defin}{Definition}[section]
\newtheorem{theorem}[defin]{Theorem}
\newtheorem{lemma}[defin]{Lemma}
\newtheorem{proposition}[defin]{Proposition}
\theoremstyle{definition} {\newtheorem{remark}[defin]{Remark}}

\title[Homogenization and $\cA$-quasiconvexity for $p=1$]{Homogenization of functionals with linear growth in the context of $\cA$-quasiconvexity}%
\author{Jos\'e Matias}
\author{Marco Morandotti}
\address{SISSA -- International School for Advanced Studies, Via Bonomea, 265, 34136 Trieste, Italy}
\email[M.~Morandotti \myenv]{marco.morandotti@sissa.it}
\author{Pedro M.~Santos}
\address{CAMGSD, Departamento de Matem\'atica, Instituto Superior T\'ecnico, Av.\@ Rovisco Pais, 1, 1049-001 Lisboa, Portugal}
\email[J.~Matias, P.~M.~Santos]{\{jmatias,pmsantos\}@math.ist.utl.pt}

\date{September 26, 2014. Preprint SISSA: 49/2014/MATE}

\makeindex%

\begin{document}

\begin{abstract}
This work deals with the homogenization of functionals with linear growth in the context of $\cA$-quasiconvexity.
A representation theorem is proved, where the new integrand function is obtained by solving a cell problem where the  coupling between homogenization and the $\cA$-free condition plays a crucial role.
This result extends some previous work to the linear case, thus allowing for concentration effects.
\end{abstract}
\maketitle%
{\small

\keywords{\noindent {\bf Keywords:} {$\cA$-quasiconvexity, homogenization, representation of integral functionals, concentration effects.}
}

\bigskip
\subjclass{\noindent {\bf {2010}
Mathematics Subject Classification:}
{
Primary 	35B27;  	
Secondary 
49J40,  	
49K20,   
35E99.   
 }}
}\bigskip
\bigskip

\tableofcontents%

\section{Introduction}
The mathematical theory of homogenization was introduced in order to describe the behavior of composite materials and reticulated structures. 
These are characterized by the fact that they contain different constituents, finely mixed in a structured way which 
bestows enhanced properties on the composite material.
Since heterogeneities are small compared with global dimensions, usually different scales are used to describe the material: a macroscopic scale describes the behavior of the bulk, while at least one microscopic scale describes the heterogeneities of the composite. In the context of Calculus of Variations the limiting (homogenized)  behavior is usually captured through $\Gamma$-convergence techniques (see \cite{DGF}, \cite{DM93}, \cite{BRA}) or through the unfolding operator (see \cite{Dm1}, \cite{Dm2}, \cite{Dm3}). An important tool to address the case of two-scale homogenization was developed in the works of G.\@ Nguetseng \cite{NG} and G.\@ Allaire \cite{Allaire} (see also \cite {BA} and \cite{LNW}).

Though there is an extensive bibliography in the subject, here we mention some results closely connected to this work, namely homogenization results for different growth or coercivity conditions of the integrand, and for integrands depending on gradients, or more generally depending on vector fields in the kernel of a constant-rank, first-order linear differential operator $\cA$, i.e., in the context of $\mathcal{A}$-quasiconvexity.
The asymptotic behavior of functionals modeling elastic materials with fine microstructure (the fineness being  accounted for  by a small parameter $\eps>0$), without considering the possibility of fracture was developed in the seminal works of S.\@ M\"uller \cite{M} and  A.\@ Braides \cite {B}.  The case involving fracture was considered in \cite{BDV}.
In the first two cases, the models involve integrands depending on the derivative of a Sobolev function whereas in the third the integrands depend on the derivative of a special function of bounded variation.
The growth conditions assumed on the integrand function are usually of order $p>1$ and $p=1$, respectively, thus allowing for concentrations to occur in the latter case. 

To ensure lower semicontinuity of the functionals, quasiconvexity is the property that the integrand functions are required to satisfy.
It is the natural generalization to higher dimension of the notion of convexity, and it is expressed by an optimality condition with respect to variations that are gradients.
This notion was further generalized by the introduction of $\cA$-quasiconvexity, and, as noted by Tartar \cite{T}, this allows to treat more general problems in continuum mechanics and electromagnetism, where constraints other than vanishing curl are considered. 

We present here examples of such operators (for more examples see \cite{FM}):
\begin{enumerate}
\item  ($\cA= \text{div}$) For $\mu \in {\cM}(\Omega; \R{N})$ we define
$$\cA \mu = \sum_{i=1}^N \frac {\partial \mu^i} {\partial x_i}.$$
\item ($\cA=\text{curl}$) For  $\mu \in {\cal M}(\Omega; \R{d})$, with $d=m \times N$, we define
$$\cA \mu = {\left( \frac {\partial \mu_k^j} {\partial x_i} -  \frac {\partial \mu_i^j} {\partial x_k} \right) }_{j=1,..,m;\, i,k=1,..,N}.$$
\item (Maxwell's Equations)  For  $\mu \in {\cM}(\R{3}; \R{3 \times 3})$ we define
$$ \cA \mu = \left( \text{div} (m+h), \text{curl} \,h \right)$$
where $\mu = (m,h).$
\end{enumerate}
The notion of $\mathcal{A}$-quasiconvexity and its implications for the lower semicontinuity of functionals was first investigated by B.\@ Dacorogna \cite{Dac} and later developed by I.\@ Fonseca and S.\@ M\"uller \cite{FM} for the study of lower semicontinuity of functionals on $\mathcal{A}$-free fields with growth $ p > 1.$ 
A.\@ Braides, I.\@ Fonseca, and G.\@ Leoni \cite{BFL} derived a homogenization result in the context of $\cA$-quasiconvexity for integrands with growth of order $ p >1$, with a microscopic scale. 
In \cite {FK}, the authors use the unfolding operator to handle a homogenization problem with $p$-growth, but no coercivity, for integrands with two scales. 

Allowing for linear growth, i.e., taking $p=1$, and coercivity implies working with sequences that are only bounded in $L^1$ and hence that can converge weakly-* (up to a subsequence) to some bounded Radon measure. In \cite{KriRin_CV_10} the relaxation of signed functionals with linear growth in the space BV of functions with bounded variation was studied. The generalization of this result in the context of $\mathcal{A}$-quasiconvexity was done in \cite{FM1}, \cite{rin1} and \cite{BCMS}.

In this work we derive a homogenization result in the context of $\cA$-quasiconvexity for integral functionals with linear growth.
This extends to the case $p=1$ the homogenization results derived in \cite{BFL}.
The linear growth condition implies that concentration effects may appear and they need to be treated by carefully applying homogenization techniques in the setting of weak-* convergence in measure.

\smallskip
\par In what follows let $\Omega\subset\R{N}$, $N\geq2$, be a bounded open set, and, for $d\geq1$,
let $f:\Omega{\times}\R{d}\to[0,+\infty)$ be a non-negative measurable function in the first variable and Lipschitz continuous in the second, that satisfies the following linear growth-coercivity condition: there exist $C_1,C_2>0$ such that
\be\label{200}
C_1|\zeta|\leq f(x,\zeta)\leq C_2(1+|\zeta|)\qquad\text{for all $(x,\zeta)\in\Omega{\times}\R{d}$}.
\ee
Moreover we assume that $x\mapsto f(x,\zeta)$ is $Q$-periodic for each $\zeta\in\R{d}$. 
Here and in the following, $Q:=(-\frac12,\frac12)^N$ will denote the unit cube. 

\par We will consider a linear first order partial differential operator $\cA:\cD'(\Omega;\R{d})\to\cD'(\Omega;\R{M})$ of the form
\be\label{operator} 
\cA = \sum_{i=1}^N A^{(i)} \frac{\partial}{\partial x_i}, \qquad A^{(i)} \in \M{M\times d}, \; \; M \in \N{},
\ee
that we assume throughout  to satisfy the following two conditions:
\begin{enumerate}
 \item[$(H_1)$] Murat's condition of {\it constant rank} (introduced in \cite{Mur_84}; see also \cite{FM}) i.e.,  there exists $c\in \N{}$ such that 
$$
\mathrm{rank}\, \left(\sum_{i=1}^N A^{(i)} \xi_i\right)=c \quad\, \text{for all}\,\,\, \xi=(\xi_1,...,\xi_d)\in \cS^{d-1};
$$
here and in the following $\cS^{d-1}$ denotes the unit sphere in $\R{d}$.
\item[$(H_2)$]  $\mathrm{Span}(\cC)=\R{d}$ (where $\cC$ stands for the characteristic cone associated with the operator $\cA$; see Definition \eqref{356}).
\end{enumerate} 

\par Let $\cM(\Omega;\R{d})$ denote the space of $\R{d}$-valued Radon measures defined on $\Omega$, and let $\mu\in\cM(\Omega;\R{d})\cap\ker\cA$. 
The object of this work is to prove a representation theorem for the functional
$$
\cF(\mu):=\inf\left\{\liminf_{n\to\infty} \int_\Omega f\left(\frac{x}{\eps_n},u_n(x)\right)\,\de x,\;\;\begin{array}{l}
\{u_n\}\subset\Lp1{}(\Omega;\R{d}),\;\;\cA u_n\stackrel{W^{-1,q}}{\longrightarrow}0, \\
u_n\wsto\mu,\,\, |u_n|\wsto\Lambda,\;\; \Lambda(\partial\Omega)=0
\end{array}\right\},
$$
where $\cA$ is defined in \eqref{operator} and satisfies $(H_1)$ and $(H_2)$, and  $q\in(1,\frac{N}{N-1})$.
%

\par Given a function $\phi:\R{d}\to\R{}$, we recall that its \emph{recession function} at infinity is defined by
\be\label{104}
\phi^\infty(b):=\limsup_{t\to+\infty}\frac{\phi(tb)}t,\qquad \text{for all $b\in\R{d}$.}
\ee

We also recall that for a measure $\mu\in\cM(\Omega;\R{d})$ we write $\mu=\mu^a+\mu^s$ for its Radon-Nikod\'ym decomposition, where $\mu^a$ is the absolutely continuous part with respect to the Lebesgue measure $\cL^N$, and $\mu^s$ is the singular part. 
This means that there exists a density function $u^a\in L^1(\Omega;\R{d})$ such that $\mu^a=u^a\cL^N$.
In the following, we will denote by $\Lp1{Q-\per}(\R{N};\R{d})$ the space of $\R{d}$-valued $L^1$ functions defined on $Q$ and extended by $Q$-periodicity to the whole of $\R{N}$.

We now state the main result of this paper.
\begin{theorem}\label{main}
Let $\Omega\subset\R{N}$ be a bounded open set, and 
let $f:\Omega{\times}\R{d}\to[0,+\infty)$ be a function which is measurable and $Q$-periodic in the first variable and Lipschitz continuous in the second, satisfying the growth condition \eqref{200}. 
For any $b\in\R{d}$, define 
\be\label{103}
f_{\cA-\hom}(b):=\inf_{R\in\N{}} \inf\left\{\ave_{RQ} f(x,b+w(x))\,\de x,\quad w\in\Lp1{RQ-\per}(\R{N};\R{d})\cap\ker\cA,\;\;\ave_{RQ} w=0\right\},
\ee
where $R\in\N{}$ and $f_{\cA-\hom}^\infty$ is the recession function of $f_{\cA-\hom}$ (see \eqref{104}).
For every $\mu\in\cM(\Omega;\R{d})\cap\ker\cA$, let $\mu=u^a\cL^N+\mu^s$ and let
$$
\cF_{\cA-\hom}(\mu):=\int_\Omega f_{\cA-\hom}(u^a)\,\de x+\int_\Omega f_{\cA-\hom}^\infty\left(\frac{\de\mu^s}{\de|\mu^s|}\right)\,\de|\mu^s|.
$$
Then, $\cF(\mu)=\cF_{\cA-\hom}(\mu)$.
\end{theorem}

\par The overall plan of this work in the ensuing sections is as
follows: in Section \ref{preliminaries} we set up the notation, concepts, and preliminary results that will be used throughout the paper. 
Section \ref{mainproof} will be devoted to the proof of Theorem \ref{main}.



\section{Preliminaries}\label{preliminaries}

\subsection{Remarks on measure theory}

\par In this section we recall some notations and well known results in Measure Theory
(see, e.g., \cite{AmbrosioFuscoPallara00}, \cite{EG}, \cite{FonLeo}, and the
references therein).

\par Let $X$ be a locally compact  metric space and let $C_{c}(X;\R{d})$,  $d\geq 1,$
denote the set of continuous functions with compact support on
$X$. We denote by  $C_{0}(X;\R{d})$  the completion of $C_{c}(X;\R{d})$ with
respect to the supremum norm.  Let $\cB(X)$ be the Borel $\sigma$-algebra of $X.$ By 
Riesz's Representation Theorem, the dual of the Banach space
$C_{0}(X;\R{d})$, denoted by $\cM (X; \R{d})$, is the space of finite $\R{d}$-valued Radon measures $\mu:
{\cB}(X)\to \R{d}$ under the pairing 
$$<\mu,\varphi>:=\int_{X} \varphi \, d\mu\equiv \sum_{i=1}^{d} \int_{X}\varphi_i \, d\mu_i,\qquad \text{for every $\varphi\in C_0(X;\R{d})$}.$$
The space $\mathcal M (X; \R{d})$ will be endowed with the weak$^*$-topology deriving from this duality. In particular we say that
a sequence $\{\mu_n\}\subset \mathcal M (X; \R{d})$ converges weakly$^*$ to $\mu \in \mathcal M (X; \R{d})$ (indicated by $\mu_{n}
\overset{*}{\rightharpoonup} \mu$) if  for all $\varphi\in
C_{0}(X;\R{d})$
$$\lim_{n\to \infty} \displaystyle\int_{X} \varphi\, d\mu_{n}=\displaystyle\int_{X} \varphi\, d\mu.$$
If $d=1$ we write by simplicity
$\mathcal M (X)$ and we denote by $\mathcal M^+ (X)$ its subset of positive measures.
 Given $\mu\in \mathcal M (X; \R{d})$, let $|\mu|\in\cM^+(X)$  denote  its {\it total
variation} and  let $\supp\;\mu$ denote  its  {\it support}.

\begin{remark}\label{inclusions}
We recall that for an open set $\Omega\subset\R{N}$ we have $W_{0}^{1,p}(\Omega;\R{d}) \subset\subset C_{0}({\Omega};\R{d})$ for  ${p}>N$. 
Moreover, $\mathcal M (\Omega; \R{d})$ is compactly embedded in $W^{-1,q}(\Omega; \R{d})$,
$1<q<\frac{N}{N-1}$, where $W^{-1,q}(\Omega; \R{d})$ denotes the dual 
 space of $W_0^{1,q'}(\Omega; \R{d})$  with $q'$, the conjugate exponent of $q$,  given by the relation 
$\frac1{q} + \frac1{q'}=1$.
\end{remark}

\par We will need the following strong form of the Besicovitch's derivation theorem which is
due to Ambrosio and Dal Maso \cite{ADMaso_92} 
(see also \cite[Theorem~2.22 and Theorem~5.52]{AmbrosioFuscoPallara00}
or \cite[Theorem~1.155]{FonLeo}).

\begin {theorem} \label{general}
Let $\mu\in \mathcal M^+(\Omega)$ and $\nu\in \mathcal M (\Omega; \R{d})$. Then
there exists a Borel set $N \subset \Omega$ with $\mu(N)=0$ such that
for every $x \in (\supp \mu)  \backslash N$
$$\frac{d\nu}{d\mu}(x) 
=\frac{d\nu^{a}}{d\mu}(x) 
= \lim_{\epsilon \rightarrow 0} 
\frac{\nu \big( D(x,\eps) \cap \Omega\big)}{\mu \big( D(x,\eps) \cap \Omega\big)}\in \R{}$$
and
$$\frac{d\nu^{s}}{d\mu}(x) 
= \lim_{\epsilon \rightarrow 0} 
\frac{\nu^s \big( D(x,\eps) \cap \Omega\big)}{\mu \big( D(x,\eps) \cap \Omega\big)}
=0,$$
where $D\subset\R{N}$ is any bounded, convex, open set containing the origin and $D(x,\eps):=x+\eps D$
(the exceptional set $N$ is independent of the choice of $D$).
\end{theorem}

\par The definition of {\it tangent measures} was originally introduced by Preiss \cite{Preiss} and is relevant for studying
 the local behaviour (or blow-up) of a given
measure. Here we give an adaptation of this notion given by Rindler \cite{rin1} (see also \cite{AmbrosioFuscoPallara00}).

\begin{defin}[Tangent measures on convex sets] \label{tan} 
Let $D\subset \R{N}$ be a bounded, convex, open set containing the origin. 
Given $\mu\in \cM (\Omega; \R{d})$ and  $x_0\in \Omega$   we define the measure $T_{*}^{(x_0,\delta)}\mu\in \cM (\overline{D}; \R{d})$, for any $\delta >0$,  by
$$<T_{*}^{(x_0,\delta)}\mu,\varphi>= \int_{{D(x_0,\delta)}} \varphi \left(\frac{x-x_0}{\delta}\right)\, d\mu, \quad 
\text{for any $\varphi\in C(\overline{D}; \R{d})$}.$$
The tangent space of  $\mu$ at $x_0$, $\Tan_{D}(\mu,x_0)$, consists of all the $\R{d}$-values measures 
$\nu \in \cM (\overline{D}; \R{d})$ which are  the weak$^*$-limit  of a rescaled sequence of measures
of the type 
$$\frac{T_{*}^{(x_0,\delta_n)}\mu}{|\mu|(D(x_0,\delta_n))}$$
for some infinitesimal sequence $\{\delta_n\}_{n\in \N{}}$.
\end{defin}

\noindent In the conditions of Definition \ref{tan}, from an adaptation of Theorem 2.44 in \cite{AmbrosioFuscoPallara00} to convex sets, it follows that 
$$
\Tan_{D}(\mu,x_0)=\frac{d\mu}{d|\mu|}(x_0)\cdot \Tan_{D}(|\mu|,x_0),
$$
whenever $\frac{d\mu}{d|\mu|}(x_0)$ exists and it is finite.
  
As noted in \cite{BCMS}, the following result holds
\begin{lemma}\label{lemmarin} 
Let $\mu\in {\cM} (\Omega; \R{d})$. Then there exists $E\subset \Omega$
with $|\mu|(\Omega\setminus E)=0$ such that for all $x_0\in E$, given $D\subset \R{N}$  an open convex set containing the origin,
there exist $\tau\in \Tan_{D}(\mu,x_0)$ with $|\tau|(D)=1$ and $|\tau|(\partial D)=0.$
\end{lemma}
\begin{remark}\label{suc-front}
\begin{itemize}
\item[(i)] The measure $\tau$ in Lemma \ref{lemmarin} is defined in an open set $U$ containing $\overline{D}$, which will 
be useful for regularization purposes.  
Moreover we have that ${\cal A}\tau=0$ in $U$ whenever ${\cal A}\mu=0$ in $\Omega$.
 \item[(ii)] Given $\Lambda \in \cM^{+}(\overline{D})$ it is possible to find a sequence $\delta_n$ with
 $\Lambda (\partial D(x_0,\delta_n))=0$ such that 
$$
\frac{T_{*}^{(x_0,\delta_n)}\mu}{|\mu|(D(x_0,\delta_n))}\wsto \tau.
$$
\end{itemize}
\end{remark}
We refer the reader to the proof of Lemma 3.1 in \cite{rin1}, from which Lemma \ref{lemmarin} and Remark \ref{suc-front} are derived.


\subsection{Remarks on $\cA$-quasiconvexity}
\par We start by recalling  the notion and some  properties of $\cA$-quasiconvex functions. This notion was   introduced by  Dacorogna \cite{Dac} following the works of Murat and Tartar 
in compensated compacteness (see \cite{Mur_84} and \cite {T}) and was further
devoloped by Fonseca and M\"{u}ller \cite{FM}  (see also Braides, Fonseca, and Leoni \cite{BFL}).
\par Let $\cA: \cD'(\Omega;\R{d}) \to  \cD'(\Omega;\R{M})$ be the first order linear differential operator defined in \eqref{operator}. 
\begin{defin}[$\cA$-quasiconvex function]\label{aquasi} 
A locally bounded Borel function $f:\R{d}\to \R{}$ is said to be  $\cA$-quasiconvex if
\[
f(v)\leq \int_Q f(v+w(x))\, \de x
\]
 for all $v\in \R{d}$ and for all  $w\in C_{Q-\per}^\infty(\R{N};\R{d})$
such that $\cA w=0$ in $\R{M}$ with $\int_Q w(x)\, \de x=0$.
\end{defin}
\begin{defin}[Characteristic cone of $\cA$]\label{355}
The characteristic cone of $\cA$ is defined by
\be\label{356}
\cC:=\left\{v\in\R{d}:\exists w\in\R{N}\setminus\{0\}, \left(\sum_{i=1}^N A^{(i)}w_i\right)v=0\right\}.
\ee
\end{defin}
Given $v\in\R{d}$, we define the linear subspace of $\R{N}$
$$
\cV_v:=\left\{w\in\R{N}:\left(\sum_{i=1}^N A^{(i)}w_i\right)v=0\right\}.
$$
If $v\notin\cC$, then $\cV_v=\{0\}$, otherwise $\cV_v$ is a non trivial subspace of $\R{N}$.

\par The following general result can be found in \cite[Lemma~2.14]{FM}. 
We will use it in the sequel applied to smooth functions.
\begin{proposition}\label{projection}
Given  $q>1$ there exists a bounded linear operator $\cP: L_{Q-\per}^q (\R{N};\R{d}) \to L_{Q-\per}^q (\R{N};\R{d})$ such that $\cA(\cP u)=0$. 
Moreover, $\int_Q \cP u=0$ and there exists a constant $C>0$ such that
$$\left|\left| u- \cP u \right|\right|_{L^q} \leq C \|\cA u\|_{W^{-1,q}}$$
for every $u \in L_{Q-\per}^q (\R{N};\R{d})$ with $\int_Q u=0.$
\end{proposition}

\subsection{Other remarks}



 Given a continuous function $f \in C(\Omega \times \R{d})$ we start by defining the linear mapping 
$T:C(\Omega \times \R{d}) \to C(\Omega \times B)$ by 
\begin{equation}\label{T}
Tf(x,\xi)=(1-|\xi|)f\left(x,\frac{\xi}{1-|\xi|}\right)
\end{equation}
and we introduce the set 
$$E(\Omega \times \R d):=\left\{f \in C(\Omega \times \R d): Tf \text{ has a continuous extension to} 
\,\, C(\overline{\Omega \times B}) \right\}.$$
For $f\in E(\Omega \times \R d)$,    $f^{\infty}$ represents the continuous extension of $Tf$, precisely, 
\begin{equation*}
f^{\infty}(x,\xi):= \lim_{\tiny \begin{array}{c}
 x^{'}\to x\\
\xi^{'}\to \xi\\
 t \to \infty
\end{array}
} \frac{f(x',t\xi')}{t}
\end{equation*}
\begin{remark}
Note that definition \eqref{104} is a generalization of this definition of recession function for  $\Phi: \R{d} \to \R{}$ with linear growth.
\end{remark}
\begin{theorem}[Reshetnyak's Continuity Theorem {\cite[Theorem~5]{KriRin_CV_10}}]\label{Reshetnyak}
Let $f\in E(\Omega;\R{d})$ and let  $\mu,\,\mu_n\in \cM (\Omega; \R{d})$ be such that
$\mu_n  
\wsto \mu$  in  $\cM (\Omega; \R{d})$   and  $\langle\mu_n\rangle(\Omega)
 \to \langle\mu\rangle (\Omega),$  where 
$$\langle\nu\rangle
:=\sqrt{1+|\nu^a|^2}\cL^N+|\nu^s|, \quad \nu=\nu^a \cL^N +\nu^s\in \cM (\Omega; \R{d}).$$
Then
$$\lim_{n\to \infty}\widehat{\cF}(\mu_n)= \widehat{\cF}(\mu)$$
where
\begin{equation}\label{Fhat}
\widehat\cF(\nu):= \int_\Omega f(x,\nu^a(x))\, dx 
+\int_\Omega f^\infty\left(x, 
\frac{d\nu^s}{d|\nu^s|}(x)\right) 
\, d|\nu^s|,\quad \nu\in \cM (\Omega; \R{d}).
\end{equation}
\end{theorem}
The following upper semicontinuity result which is a consequence of Theorem \ref{Reshetnyak} will be useful in the proof of our main result Theorem \ref{main}. For a proof see \cite[Corollary 2.11]{BCMS}.
\begin{proposition} \label{usc}
Let $f: \Omega \times \R{d} \to \R{}$ be a continuous function such that there exists $C>0$ with
$$|f(x,\xi)|\leq C(1+|\xi|),\,\, \text{ for all  $(x,\xi) \in \Omega \times \R{d}$.}$$
Let  $\mu,\,\mu_n\in \mathcal M (\Omega; \R{d})$ be such that
$\mu_n \wsto \mu$  in  $\mathcal M (\Omega; \R{d})$   and  $\langle\mu_n\rangle(\Omega) \to \langle\mu\rangle (\Omega).$
Then
\begin{equation*}
\widehat\cF(\mu)\geq \limsup_{n\to \infty} \widehat\cF(\mu_n)
\end{equation*}
where $\widehat\cF$ was defined in \eqref{Fhat}.
\end{proposition}

\subsection{Preliminary results}

 Fix some sequence $\{ \eps_n \}$ and let $\cO(\Omega)$ denote the collection of all open subsets of $\Omega$.
For $A\in\cO(\Omega)$, the localized version of $\cF$, denoted by $\cF_{\{\eps_n\}}(\mu; A)$, is defined by
\be\label{202}
\cF_{\{\eps_n\}} (\mu;A) := \inf \left\{ \liminf_{n\to\infty} \int_A f\left(\frac{x}{\eps_n},u_n(x)\right)\,\de x,\begin{array}{l}
\{u_n\}\subset\Lp1{}(\Omega;\R{d}),\;\;\cA u_n\stackrel{W^{-1,q}}{\longrightarrow}0, \\
u_n\wsto\mu,\,\, |u_n|\wsto\Lambda,\;\; \Lambda(\partial\Omega)=0
\end{array}\right\}.
\ee 
Note that using the regularization of $\mu$ as test sequence, that is $u_n = \rho_{\eps_n} * \mu$, by \eqref{200} we obtain the upper bound
$$
\mathcal{F}_{\{\eps_n\}} (\mu;A) \leq C \left(|A|+ |\mu|(\overline{A})\right).
$$


We begin by proving the following subadditivity result:
\begin{proposition}[Subadditivity for nested sets]\label{sub}
For any $A\subset\subset B\subset\subset C\subset\Omega$ there holds
\be\label{210}
\cF_{\{ \eps_n^1 \}}(\mu;C)\leq \cF_{\{ \eps_n^1 \}}(\mu;B)+\cF_{\{ \eps_n^1 \}}(\mu;C\setminus\cl A).
\ee
where $\{\eps_n^1\}\subset\{\eps_n\}$ is an appropriate subsequence.
\end{proposition}
\begin{proof}
Fix $\delta>0$. Consider a subsequence $\{\eps_n^1\}\subset\{\eps_n\}$ and an associated sequence $\{v_n^1\}\subset\Lp1{}(\Omega;\R{d})$ such that $v_n^1 \wsto \mu$, $\cA v_n^1 \to 0$ and
\be\label{211}
\cF_{\{ \eps_n \}}(u;B) + \delta \geq \lim_{n\to+\infty} \int_B f\left(\frac{x}{\eps_n^1},v_n^1\right)\de x, 
\ee
By extracting another subsequence  $\{\eps_n^2\}\subset\{\eps_n^1\}$ there exists a sequence $\{w_n^2\}\subset\Lp1{}(\Omega;\R{d})$ such that $w_n^2 \wsto \mu$, $\cA w_n^2 \to 0$ and
\be\label{211a}
\cF_{\{ \eps_n^1 \}}(u;C \setminus \cl A) + \delta \geq \lim_{n\to+\infty} \int_{C \setminus \cl A}  f\left(\frac{x}{\eps_n^2},w_n^2\right)\de x.
\ee
We denote by $\{ v_n^2 \}$ the subsequence of $\{ v_n^1 \}$ corresponding to $\{ \eps_n^2 \}$.

Consider smooth cut-off functions $\phi_j$ such that, for all $j\in\N{}$, $\phi_j\equiv1$ on $A$, $\phi_j\equiv0$ on $C\setminus\cl B$, and such that
\be\label{212}
\lim_{j \to +\infty} \limsup_{n \to +\infty} \int_{ \{ 0 <  \phi_j < 1 \}} \left(1+ |v_n^2|+|w_n^2| \right) \de x=0.
\ee
Property \eqref{212} ensures that $|\{0<\phi_j<1\}|\to0$ as $j\to\infty$ and that $\{v_n^2\}$ and $\{w_n^2\}$ do not have concentrations in the transition layer of $\phi_j$.

Define $ U_{n,j}:=(1-\phi_j)w_n^2+\phi_j v_n^2$. It is not difficult to show that $\{ U_{n,j}\}\subset L^1(\Omega;\R{d})$, and that $U_{n,j} \wsto \mu$, $\cA U_{n,j}\to0$ in $W^{-1,q}(\Omega;\R{M})$ as $n \to \infty$.
Let $A_j:=\{x\in\Omega:\phi_j(x)\equiv1\}$ and $B_j:=\{x\in\Omega:\phi_j(x)\equiv0\}$ and notice that $A\subset A_j\subset B$ and $C\setminus\cl B\subset B_j\subset C\setminus\cl A$.
We have
\begin{equation*}
\begin{split}
\cF_{\{ \eps_n^1\}}(u;C) \leq & \liminf_{j \to +\infty} \liminf_{n\to+\infty} \int_C f\left(\frac{x}{\eps_n^2},U_{n,j}\right)\de x \\
\leq & \limsup_{j \to +\infty} \limsup_{n\to+\infty} \int_{A_j} f\left(\frac{x}{\eps_n^2},v_n^2\right)\de x + \limsup_{j \to +\infty} \limsup_{n\to+\infty} \int_{B_j} f\left(\frac{x}{\eps_n^2},w_n^2\right)\de x \\
       & +\limsup_{j \to +\infty} \limsup_{n\to+\infty}\int_{\{0<\phi_j<1\}} f\left(\frac{x}{\eps_n^2},U_{n,j} \right)\de x\\
\leq & \lim_{n\to+\infty} \int_{B} f\left(\frac{x}{\eps_n^2},v_n^2\right)\de x + \lim_{n\to+\infty} \int_{C\setminus\cl A} f\left(\frac{x}{\eps_n^2},w_n^2\right)\de x \\
       & +\limsup_{j \to +\infty} \limsup_{n\to+\infty} \int_{\{0<\phi_j<1\}} C(1+|U_{n,j}|)\de x\\
\leq & \cF_{\{ \eps_n^1 \}}(u;B)+\cF_{\{ \eps_n^1 \}}(u;C\setminus\cl A) + 2 \delta
\end{split}
\end{equation*}
where we used that $f\geq0$, \eqref{211}, \eqref{211a}, and \eqref{212}. 
Letting $\delta \to 0$ yields \eqref{210}.
\end{proof}

\begin{proposition}\label{measure}
Given $\mu \in \cM(\Omega;\R{d})$ and a sequence $\{\eps_n\}$, we can find a subsequence $\{ \eps_n^1 \} \subset \{ \eps_n \}$, a sequence $u_n^1 \wsto \mu$, $\cA u_n^1 \to 0$, and a bounded Radon measure $\Phi_{\mu}$ such that 
\be\label{2061}
f \left( \frac {x}{\eps_n^1}, u_n^1(x) \right)\de x \wsto \Phi_\mu.
\ee
Moreover, 
for every open set $A \subset \subset \Omega$, we can find a further subsequence $\{\eps_n^2\} \subset \{ \eps_n^1 \}$ such that
\be\label{206}
\cF_{\{\eps_n^2\}}(\mu;A)=\Phi_\mu(\overline A).
\ee
\end{proposition} 
\begin{proof}
Fix $\mu\in\cM(\Omega;\R{d})$ and $\delta>0$.  Then there exists  a subsequence $\{\eps_n^1\}\subset\{\eps_n\}$ and  a sequence $\{u_n^1\}\subset L^1(\Omega;\R{d})$ such that
\be\label{205}
\cF_{\{\eps_n\}}(\mu) + \delta \geq \lim \int_\Omega f\left(\frac{x}{\eps_n^1},u_n^1(x)\right)\,\de x
\ee
and therefore by \eqref{200} we have that
$$
f\left(\frac{x}{\eps_n^1},u_n^1(x)\right)\,\de x\wsto\Phi_\mu,
$$
for some bounded Radon measure $\Phi_\mu$.
We now claim that we can extract another subsequence $\{\eps_n^2\}\subset\{\eps_n^1\}$ for which \eqref{206} holds.
We start by noting that
\be\label{207}
\cF_{\{\eps_n^2\}}(\mu;A)\leq \liminf_{n\to\infty} \int_A f\left(\frac{x}{\eps_n^2},u_n^2(x)\right)\,\de x\leq \limsup_{n\to\infty} \int_{\cl A} f\left(\frac{x}{\eps_n^1},u_n^1(x)\right)\,\de x\leq \Phi_\mu(\cl A),
\ee
for every subsequence $\{(\eps_n^2,u_n^2)\}\subset\{(\eps_n^1,u_n^1)\}$, where the last inequality follows from known facts about measure theory (see, e.g., \cite{FonLeo}).
To prove the opposite inequality, fix $\delta>0$ and choose an open set $B\subset A$ such that $\Phi_\mu(\partial B)=0$ and $\Phi_\mu(\cl A\setminus B)<\delta$.
Now,
\be\label{208}
\begin{split}
\Phi_\mu(\cl A) & \leq \Phi_\mu(B)+\delta \\
& = \Phi_\mu(\Omega)-\Phi_\mu(\Omega\setminus\cl B) +\delta\\
& \leq \cF_{\{\eps_n^2\}}(\mu) + \delta -\cF_{\{\eps_n^2\}}(\mu;\Omega\setminus\cl B)+\delta,
\end{split}
\ee
where the last line follows by applying \eqref{207} to the closure of $\Omega\setminus\cl B$, and, by \eqref{205} and \eqref{210} the sequence $\{\eps_n^2\}$ is chosen in such a way that
\be\label{209}
\cF_{\{\eps_n^2\}}(\mu)\leq \cF_{\{\eps_n^2\}}(\mu;A)+\cF_{\{\eps_n^2\}}(\mu;\Omega\setminus\cl B).
\ee
By plugging \eqref{209} in \eqref{208}, and letting $\delta\to0$, we conclude the proof of \eqref{206}. 
\end{proof}

\par In the proof of the main result we will use the following properties of $f_{\cA-\hom}$.
Recall its definition in \eqref{103}.

\begin{proposition}\label{Lipschitz}
Let $f:\Omega\times\R{d}\to\R{}$ be Lipschitz continuous in the second variable with Lipschitz constant $L>0$.
Then the function $f_{\cA-\hom}:\R{d}\to\R{}$ defined in \eqref{103} is Lipschitz continuous with the same constant.
\end{proposition}
\begin{proof}
Let $b_1,b_2\in\R{d}$ and let $\eps>0$. 
From the definition of $f_{\cA-\hom}$, there exist $k_1\in\N{}$ and $w_1\in L^1_{kQ-\per}(\R{N};\R{d})\cap\ker\cA$ such that $\ave_{k_1Q} w_1=0$ and 
$$\ave_{k_1Q} f(y,b_1+w_1(y))\,\de y\leq f_{\cA-\hom}(b_1)+\eps.$$
Again by the definition of $f_{\cA-\hom}$ and by the previous inequality we have
\begin{equation*}
\begin{split}
f_{\cA-\hom}(b_2)-f_{\cA-\hom}(b_1) \leq & \ave_{k_1Q} f(y,b_2+w_1(y))\,\de y-\ave_{k_1Q} f(y,b_1+w_1(y))+\eps \\
\leq & \ave_{k_1Q} |f(y,b_2+w_1(y))-f(y,b_1+w_1(y))|\,\de y +\eps \\
\leq & L|b_2-b_1|+\eps.
\end{split}
\end{equation*}
The result follows by exchanging the roles of $b_1$ and $b_2$ and by letting $\eps\to0$.
\end{proof}

\begin{proposition}[Invariance under translations]\label{translations}
Let $b \in \R{d}$ and $\gamma\in Q$. Define
\be\label{222}
f_{\cA-\hom}^\gamma(b):=\inf_{R\in\N{}}\inf\left\{\ave_{RQ} f(y+\gamma,b+w(y))\de y,\;\; w\in L^1_{RQ-\per}(\R{N};\R{d})\cap\ker\cA,\ave_{RQ} w=0\right\}.
\ee
Then we have that $f_{\cA-\hom}(b) = f_{\cA-\hom}^\gamma(b)$.
\end{proposition}

\begin{proof}
The proof is a simple computation. 
Let $w\in L^1_{RQ-\per}(\R{N};\R{d})\cap\ker\cA$ with $\ave_{RQ} w=0$.
Then,
\be\label{223}
\begin{split}
\ave_{RQ} f(y+\gamma,b+w(y))\,\de y = & \ave_{RQ+\gamma} f(x,b+w(x-\gamma))\de x \\
= & \ave_{RQ} f(x,b+w(x-\gamma))\de x \\
& + \ave_{(RQ+\gamma)\setminus RQ} f(x,b+w(x-\gamma))\de x \\
& -\ave_{RQ\setminus(RQ+\gamma)} f(x,b+w(x-\gamma))\de x \\
=&  \ave_{RQ} f(x,b+w_\gamma(x))\de x \\
\geq & f_{\cA-\hom}(b),
\end{split}
\ee
where we defined $w_\gamma(x):=w(x-\gamma)$.

We notice that $w_\gamma$ is $RQ$-periodic, $\ave_{RQ} w_\gamma=\ave_{RQ} w=0$, and $\cA w_\gamma=0$, which means that $w_\gamma$ is a competitor for \eqref{103}.
The third equality in \eqref{223} holds because the second and third integrals cancel out.
Indeed, by using the $RQ$-periodicity of $w$ and noticing that $f$ is also $RQ$-periodic in the first variable, by means of shifts by $R$ along the coordinate directions one can easily proof that cancellation occurs.

By taking the infimum over all $w\in L^1_{RQ-\per}(\R{N};\R{d})\cap\ker\cA$ with $\ave_{RQ} w=0$, and the infimum over $R\in\N{}$, we obtain
$$
f_{\cA-\hom}^\gamma(b)\geq f_{\cA-\hom}(b).
$$
The reverse inequality is obtained in the same way.
\end{proof}

The following result will be used in Section \ref{LBsp}.
It is an adaptation \cite[Lemma 2.20]{BCMS}.
\begin{lemma}\label{375}
Let $f_n:\Omega\times\R{d}\to\R{}$ be a family of Lipschitz continuous functions in the second variable with the same Lipschitz constant $L$ and let $\nu\in\cS^{N-1}$.
Let $\{u_n\},\{v_n\}\subset L^q(Q_\nu;\R{d})$ be sequences such that $u_n-v_n\wsto0$ in $\cM(Q_\nu;\R{d})$ and $|u_n|+|v_n|\wsto\Lambda$ in $\cM^+(\cl Q_\nu)$, with $\Lambda(\partial Q_\nu)=0$, and $\cA(u_n-v_n)\to0$ in $W^{-1,q}(Q_\nu;\R{M})$ for some $1<q<\frac{N}{N-1}$.
Then there exists a sequence $\{z_n\}\subset L_{Q_\nu-\per}^q(\R{N};\R{d})$ such that $\int_{Q_\nu} z_n=0$, $\cA z_n=0$, $z_n\wsto0$, and
\be\label{376}
\liminf_{n\to\infty} \int_{Q_\nu} f_n(x,u_n(x))\,\de x\geq \liminf_{n\to\infty} \int_{Q_\nu} f_n(x,z_n(x)+v_n(x))\,\de x.
\ee
\end{lemma}
\begin{proof}
The proof is the same as the proof of \cite[Lemma 2.20]{BCMS}.
It is achieved by proving \eqref{376} for the unit cube $Q$ and then for the rotated cube $Q_\nu$.
The proof relies on some Lipschitz estimates which involve only the second variable of  $f$.
\end{proof}

\section{Main result}\label{mainproof}

\par In this section we prove Theorem \ref{main}. 
We shall split the proof in several steps.
\subsection{Upper bound}
In this section we prove the upper bound inequality in Theorem \ref{main}.
We start with the estimate for regular measures $\mu = u\cL^N$, i.e., we prove that:
\be\label{214}
\frac{\de\Phi_u}{\de\cL^N}(x_0)\leq f_{\cA-\hom}(u(x_0)),\qquad\text{for $\cL^N$-a.e.\@ $x_0\in\Omega$}.
\ee
Let $x_0$ be a Lebesgue point for the function $u$.
Given a vanishing sequence of radii $\{r_j\}$ such that $\Phi_u(\partial Q(x_0,r_j))=0$ for all $j\in\N{}$,  there exists a sequence $\{\eps_n^j\}$ such that
\be\label{218}
\Phi_u(Q(x_0,r_j))=\cF_{\{\eps_n^j\}}(u;Q(x_0,r_j))\qquad\text{for all $j\in\N{}$.}
\ee
We will contruct an appropriate sequence $\{u_n^j\}\subset L^1(Q(x_0,r_j);\R{d})$ such that 
$$
\left.
\begin{array}{l}
u_n^j\wsto u \\
\cA u_n^j=0
\end{array}
\right\}\qquad \text{as $n\to\infty$, for all $j\in\N{}$.}
$$
To this end, consider $k\in\N{}$ and a function $w\in L^1_{kQ-\per}(\R{N};\R{d})\cap\ker\cA$, such that $\ave_{kQ} w=0$, and let $w_n^j(x):=w(h_nk(x-x_0)/r_j)$, with $h_n\in\N{}$ to be determined.
It is not hard to prove that 
$$
\cA w_n^j=0\;\text{for all $n,j\in\N{}$}\qquad\text{and that}\;\ave_{Q(x_0,r_j)} w_n^j=\ave_{kQ} w=0.
$$
Moreover $u_n^j(x):=u(x)+w_n^j(x)$ satisfies
$$
u_n^j\wsto u,\qquad \cA u_n^j=\cA u=0\; \text{for all $n,j\in\N{}$.}
$$
We then have
\begin{equation*}
\begin{split}
\frac{\de\Phi_u}{\de\cL^N}(x_0) = & \lim_{j\to\infty}\frac{\Phi_u(Q(x_0,r_j))}{r_j^N} \stackrel{\eqref{218}}{=} \lim_{j\to\infty} \frac{\cF_{\{\eps_n^j\}}(u;Q(x_0,r_j))}{r_j^N}\\
\leq & \limsup_{j\to\infty}\liminf_{n\to\infty}\frac1{r_j^N}\int_{Q(x_0,r_j)} f\left(\frac{x}{\eps_n^j},u_n^j(x)\right)\de x \\ 
= & \limsup_{j\to\infty}\liminf_{n\to\infty}\frac1{r_j^N}\int_{Q(x_0,r_j)} f\left(\frac{x}{\eps_n^j},u(x)+w\left(h_nk\frac{x-x_0}{r_j}\right)\right)\de x \\
\leq & \limsup_{j\to\infty}\left[ \limsup_{n\to\infty} \frac1{r_j^N}\int_{Q(x_0,r_j)} f\left(\frac{x}{\eps_n^j},u(x_0)+w\left(h_nk\frac{x-x_0}{r_j}\right)\right)\de x + A_j \right]\\
= & \limsup_{j\to\infty}\limsup_{n\to\infty}\frac1{(h_nk)^N}\int_{h_nkQ} f\left(\frac{r_jy}{h_nk\eps_n^j}+\frac{x_0}{\eps_n^j},u(x_0)+w(y)\right)\de y
\end{split}
\end{equation*}
where in the fourth line we have used the Lipschitz continuity with respect to the second variable and the fact that $u$ has a Lebesgue point at $x_0$ to obtain that $A_j\to0$, and in the fifth line we changed variables $y:=h_nk\frac{x-x_0}{r_j}$.

We now choose $h_n:=\left[\frac{r_j}{k\eps_n^j}\right]$, and write
$$
\frac{r_jy}{h_nk\eps_n^j}=y+\frac1{h_n}\left\langle\frac{r_j}{k\eps_n^j}\right\rangle y;
$$
also note that, since $\left\langle\frac{r_j}{k\eps_n^j}\right\rangle\in Q$, then
\be\label{2201}
\lim_{n\to\infty} \frac1{h_n}\left\langle\frac{r_j}{k\eps_n^j}\right\rangle=0.
\ee
Let $\gamma_n^j:=\left\langle\frac{x_0}{\eps_n^j}\right\rangle$. 
By using the $Q$-periodicity of $f$ in the first variable, the $kQ$-periodicity of $w$, and \eqref{2201}, we then have that
\begin{equation*}
\begin{split}
\frac{\de\Phi_u}{\de\cL^N}(x_0)\leq & 
\limsup_{j\to\infty}\limsup_{n\to\infty}\frac1{(h_nk)^N}\int_{h_nkQ} f\left(y+\frac1{h_n}\left\langle\frac{r_j}{k\eps_n^j}\right\rangle y+\gamma_n^j,u(x_0)+w(y)\right)\de y \\
= & \limsup_{j\to\infty}\limsup_{n\to\infty}\frac1{k^N}\int_{kQ} f\left(y+\frac1{h_n}\left\langle\frac{r_j}{k\eps_n^j}\right\rangle y+\gamma_n^j,u(x_0)+w(y)\right)\de y \\
= & \limsup_{j\to\infty}\frac1{k^N}\int_{kQ} f(y+\gamma_{n_j}^j,u(x_0)+w(y))\de y \\
\leq & \limsup_{j\to\infty}\left(f_{\cA-\hom}^{\gamma_{n_j}^j}(u(x_0))+\frac1j\right)= f_{\cA-\hom}(u(x_0))
\end{split}
\end{equation*}
where in the last line we used the definition of $\inf$ and Proposition \ref{translations}. 
This proves \eqref{214}. 

\par Let now $\mu\in\cM$ be a general measure. 
Then by defining $u_n:=\rho_n*\mu$ we construct a sequence of smooth functions such that
$$
u_n\wsto\mu,\qquad\qquad\bracket{u_n}(\Omega)\to\bracket\mu(\Omega).
$$

By Proposition \ref{usc} we the have that
\begin{equation*}
\begin{split}
\cF_{\{\eps_n\}}(\mu)\leq & \liminf_{n\to+\infty} \cF_{\{\eps_n\}}(u_n)\leq\liminf_{n\to+\infty} \int_\Omega f_{\cA-\hom}(u_n(x))\,\de x \\
\leq & \limsup_{n\to+\infty} \int_\Omega f_{\cA-\hom}(u_n(x))\,\de x\leq \limsup_{n\to\infty} \cF_{\cA-\hom}(u_n) \leq \cF_{\cA-\hom}(\mu).
\end{split}
\end{equation*}
This proves that
$$
\cF(\mu)\leq\cF_{\cA-\hom}(\mu).
$$

\subsection{Lower bound -- Absolutely continuous part}
Let $\mu\in\cM(\Omega;\R{d})$,  let $\{\eps_n\}$ be a vanishing sequence and let $\Phi_\mu$ be the measure provided by the first statement in Proposition \ref{measure}, with the respective
${u_n}\subset L^1(\Omega;\R{d})$ (not relabeled) such that $u_n\wsto\mu$, and $\cA u_n\to0$ in $W^{-1,q}$.

%
%
%
%
%
%
We now prove the inequality
\be\label{311}
\frac{\de\Phi_\mu}{\de\cL^N}(x_0)\geq f_{\cA-\hom}(u(x_0)),
\ee
where we decomposed $\mu=\mu^a+\mu^s=u\cL^N+\mu^s$, and $x_0$ is a Lebesgue point for $\mu$. In particular, $\mu^s(\{x_0\})=0$.
If the Radon-Nikod\'ym derivative in the left-hand side of \eqref{311} is $\infty$ there is nothing to prove, so we shall assume that it is bounded.

By \eqref{2061} and Theorem \ref{general} we have that
$$
\frac{\de\Phi_{\mu}}{\de\cL^N}(x_0)= \lim_{j\to\infty}\lim_{n\to\infty}\frac1{r_j^N}\int_{Q(x_0;r_j)} f\left(\frac{x}{\eps_n},u_{n}(x)\right)dx.
$$
For any $x\in Q$, define $w_{n,j}(x):=u_{n}(r_jx+x_0)-u(x_0)$. 
Then we have $w_{n,j}\wsto0$ since $x_0$ is a Lebesgue point, $w_{n,j}\in L^1$, $\cA w_{n,j}\to0$ in $W^{-1,q}$. 
Notice that it is not restrictive to choose the sequence $\{r_j\}$ in such a way that 
\be\label{321}
|w_{n,j}|\wsto\lambda\qquad\text{and} \qquad \lambda(\partial Q)=0.
\ee
So,
\begin{equation*}
\begin{split}
\frac{\de\Phi_{\mu}}{\de\cL^N}(x_0)= & \lim_{j\to\infty}\lim_{n\to\infty}\frac1{r_j^N}\int_{Q(x_0;r_j)} f\left(\frac{x}{\eps_n},u(x_0)+w_{n,j}\left(\frac{x-x_0}{r_j}\right)\right)\,\de x,\\
=& \lim_{j\to\infty}\lim_{n\to\infty}\int_{Q} f\left(\frac{r_jy+x_0}{\eps_n},u(x_0)+w_{n,j}(y)\right)\,\de y,
\end{split}
\end{equation*}
where we have changed variables $y:=(x-x_0)/r_j$.
Now we can diagonalize to obtain a sequence $\hat w_k\in L^1$, such that $\hat w_k\wsto0$, $\cA\hat w_k\to0$ in $W^{-1,q}$.
So, by invoking the $Q$-periodicity of $f$ in the first variable,
$$
\frac{\de\Phi_{\mu}}{\de\cL^N}(x_0)= \lim_{k\to\infty}\int_{Q} f(s_ky+\gamma_{k},u(x_0)+\hat w_k(y))\,\de y,
$$
with $s_k:=r_{j(k)}/\eps_{n(k)}$ in such a way that $\lim_{k\to\infty} s_k=\infty$, and $\gamma_{k}:=\langle x_0/\eps_{n(k)}\rangle$.
We claim we can find a sequence $w_k$ such that
\be\label{315}
\lim_{k\to\infty} \int_Q f( s_ky + \gamma_{k}, u(x_0) + \hat{w}_k(y)) \, \de y \geq \lim_{k\to\infty} \int_Q f(s_ky + \gamma_{k}, u(x_0) + w_k(y))\, \de y,
\ee
and $ w_k\wsto0, \; \cA w_k = 0, \; \ave_Q w_k = 0, \; w_k \in L^1_{Q-{\per}}(\R{N};\R{d}).$
To this end, consider a cut-off function $\theta_i$ such that $\theta_i\equiv1$ in $Q_i:= (1-1/i)Q$ and $\theta_i\equiv0$ in $\R{N}\setminus\cl Q$. 
Then, by using Lispschitz continuity of $f$ we obtain
\be\label{316}
\begin{split}
\frac{\de\Phi_{\mu}}{\de\cL^N}(x_0)=&\lim_{k\to\infty} \int_Q  f( s_ky + \gamma_{k}, u(x_0) + \hat{w}_k(y)) \,\de y  \\
 =&\lim_{k\to\infty} \int_Q f( s_ky + \gamma_{k}, u(x_0) + \theta_i(y)\hat{w}_k(y)+(1-\theta_i)\hat w_k(y)) \,\de y\\
 \geq &\liminf_{i\to\infty}\liminf_{k\to\infty} \left[ \int_Q f( s_ky + \gamma_{k}, u(x_0) + \theta_i(y)\hat{w}_k(y)) \,\de y - I_{i,k} \right],
\end{split}
\ee
where
$$
I_{i,k}:= L \int_{Q\setminus Q_i} ( 1 - \theta_i)|\hat{w}_k(y)| \,\de y.
$$
Notice now that $\lim_{i\to\infty}\lim_{k\to\infty} I_{i,k}=0$.
Indeed, since $\hat w_k$ comes from the sequence $w_{n,j}$ by \eqref{321} we have that $|\hat w_k|\wsto\lambda$. 
By definition of weak convergence, and again by \eqref{321}, we have
\be\label{318}
\lim_{i\to\infty}\lim_{k\to\infty} I_{i,k}\leq\lim_{i\to\infty} L\lambda\big(\cl{Q\setminus Q_i}\big)=L\lambda(\partial Q)=0.
\ee

Define now $\hat w_{i,k}:=\theta_i\hat w_k$, extended by $0$ on the whole $\R{N}$. 
Notice that
\be\label{319}
\lim_{k\to\infty}\ave_Q \hat w_{i,k}=
\lim_{k\to\infty}\ave_Q \theta_i\hat w_k=0. 
\ee
From \eqref{316} and \eqref{318} 
we obtain
\be\label{320}
\begin{split}
\frac{\de\Phi_{\mu}}{\de\cL^N}(x_0)=&\lim_{k\to\infty} \int_Q   f( s_ky + \gamma_{k}, u(x_0) + \hat{w}_k(y)) \,\de y \\
\geq & \liminf_{i\to\infty}\liminf_{k\to\infty} \int_Q f(s_{k}y+\gamma_{k},u(x_0)+\hat w_{i,k}(y))\,\de y \\
= & \liminf_{i\to\infty}\liminf_{k\to\infty} \frac1{s_{k}^N} \int_{s_{k}Q} f\left(x+\gamma_{k},u(x_0)+\hat w_{i,k}\left(\frac{x}{s_{k}}\right)\right)\de x \\
\geq & \liminf_{i\to\infty}\liminf_{k\to\infty} \frac1{s_{k}^N} \int_{([s_{k}]+1)Q} f\left(x+\gamma_{k},u(x_0)+\hat w_{i,k}\left(\frac{x}{s_{k}}\right)\right)\de x \\
& -\limsup_{k\to\infty}\frac{C_1}{s_{k}^N} |([s_{k}]+1)Q\setminus s_{k}Q| |u(x_0)|, 
\end{split}
\ee
where we have used the growth condition \eqref{200} together with the fact that $\hat w_{i,k}\left(\frac\cdot{s_{k}}\right)=0$ outside $s_{k}Q$.
It is easy to see that the last term in \eqref{320} vanishes.

Define $n_k:=[s_{k}]+1$ 
and notice that we can restrict $\hat w_{i,k}\left(\frac\cdot{s_{k}}\right)$ to $n_kQ$ and extend it by $n_kQ$-periodicity to the whole $\R{N}$. Let $U_{i,k}\left(\frac\cdot{n_k}\right)$ be the extended function. 
Then \eqref{320} becomes
\begin{equation*}
\begin{split}
\frac{\de\Phi_{\mu}}{\de\cL^N}(x_0)\geq & \liminf_{i\to\infty}\liminf_{k\to\infty} \frac1{n_k^N} \int_{n_kQ} f\left(x+\gamma_{k},u(x_0)+U_{i,k}\left(\frac{x}{n_{k}}\right)\right)\de x \\
=& \liminf_{i\to\infty} \liminf_{k\to\infty}\int_Q f(n_{k}y+\gamma_{k},u(x_0)+U_{i,k}(y))\,\de y.
\end{split}
\end{equation*}
Notice that $U_{i,k}$ is a $Q$-periodic function such that $U_{i,k}\wsto0$ as $k\to\infty$, $\lim_{i\to\infty}\lim_{k\to\infty}\cA U_{i,k}=0$ in $W^{-1,q}$, and, by \eqref{319},
\be\label{323a}
\ave_Q U_{i,k}\to0\qquad \text{as $k\to\infty$}.
\ee
%
Therefore, we can apply Proposition \ref{projection} to the function $U_{i,k}-\ave_Q U_{i,k}$, to obtain a new function $V_{i,k}:=\cP(U_{i,k}-\ave_Q U_{i,k})$ such that 
\be\label{324}
\begin{split}
\frac{\de\Phi_{\mu}}{\de\cL^N}(x_0)\geq & \liminf_{i\to\infty}\liminf_{k\to\infty} \int_Q f\left(n_{k}y+\gamma_{k},u(x_0)+V_{i,k}(y)+\ave_Q U_{i,k}\right)\de y \\
= &  \liminf_{i\to\infty} \int_Q f(n_{k_i}y+\gamma_{k_i},u(x_0)+V_{i,k_i}(y))\de y, \\
\end{split}
\ee
where we have used the Lipschitz continuity of $f$ in the second variable, \eqref{323a}, and chosen an appropriate diagonalizing sequence.
By defining $w_i:=V_{i,k_i}$, formula \eqref{315} is proved.
Estimate \eqref{311} now follows upon changing variables and using Proposition \ref{translations}.
Indeed, from \eqref{324}
\begin{equation*}
\begin{split}
\frac{\de\Phi_{\mu}}{\de\cL^N}(x_0)\geq & \liminf_{i\to\infty} \frac1{n_{k_i}^N} \int_{n_{k_i}Q} f\left(x+\gamma_{k_i},u(x_0)+w_i\left(\frac{x}{n_{k_i}}\right)\right) \de x \\
\geq & \liminf_{i\to\infty} f_{\cA-\hom}^{\gamma_{k_i}}(u(x_0))=\liminf_{i\to\infty} f_{\cA-\hom}(u(x_0))=f_{\cA-\hom}(u(x_0)).
\end{split}
\end{equation*}

\begin{remark}[$\cA$-quasiconvexity of $f_{\cA-\hom}$]\label{fAhomAqc}
As a consequence of \eqref{214} and \eqref{311} we obtain that $f_{\cA-\hom}$ is $\cA$-quasiconvex.
This result will be used in Section \ref{LBsp}.
To prove it, let $b\in\R{d}$, let $w\in C^\infty_{Q-\per}(\R{N};\R{d})\cap\ker\cA$, and define $w_n(x):=w(nx)$. 
Then it is easy to see that $\cA w_n=0$, and, by the Riemann-Lebesgue Lemma, $w_n\wsto0$.
Then,
\begin{equation*}
\begin{split}
f_{\cA-\hom}(b)= & \int_Q f_{\cA-\hom}(b)\,\de x = \cF(b;Q) \leq  \liminf_{n\to\infty} \cF(b+w_n;Q) \\
\leq & \liminf_{n\to\infty} \int_Q f_{\cA-\hom}(b+w_n(x))\,\de x= \liminf_{n\to\infty} \frac1{n^N}\int_{nQ} f_{\cA-\hom}(b+w(y))\,\de y \\
= & \int_Q f_{\cA-\hom}(b+w(y))\,\de y,
\end{split}
\end{equation*}
where the second equality follows from \eqref{214} and \eqref{311}, and where we have used the lower semicontinuity of $\cF$ with respect to the weak-* convergence.
This proves that $f_{\cA-\hom}$ is $\cA$-quasiconvex.
\end{remark}

\subsection{Lower bound -- Singular part}\label{LBsp}
We now prove the inequality
\be\label{351}
\frac{\de\Phi_\mu}{\de|\mu^s|}(x_0)\geq f_{\cA-\hom}^\infty\left(\frac{\de\mu^s}{\de|\mu^s|}(x_0)\right)\qquad \text{for $|\mu^s|$-a.e.\@ $x_0\in\Omega$.}
\ee
Let $x_0\in\supp|\mu^s|\cap E$, where $E$ is the set given by Lemma \ref{lemmarin}. 
%
%
%
%
Now call 
$$v_{x_0}:=\frac{\de\mu^s}{\de|\mu^s|}(x_0),$$
that we suppose to be finite together with $\frac {\de \Lambda} {\de |\mu^s|}(x_0)$.
We distinguish two cases, namely $v_{x_0}\in\cC$ and $v_{x_0}\notin\cC$.

\textbf{Case 1: $v_{x_0}\in\cC$.} 
Let $\{\omega_1,...,\omega_k\}$ be an orthonormal basis for ${\cal V}_{x_0}$ and let $\{\omega_{k+1},...,\omega_N\}$ be an orthogonal basis for the orthogonal complement of ${\cal V}_{x_0}$ in $\R{N}$. We denote by $Q_{0}$ the unitary cube with center at $0$ and faces orthogonal to $\omega_1,\ldots,\omega_k,\ldots,\omega_N$, that is
$$Q_{0}:= \left\{ x \in \R{N}: |x\cdot \omega_i| < \frac1 {2},\, i=1,\ldots,N \right\}.$$
Choose a decreasing sequence of positive radii $r_j \to 0$ as $j\to\infty$ such that  $\Lambda(\partial Q_{0} (x_0,r_j))=0$, thus we also have $\Phi_\mu(\partial Q_{0} (x_0,r_j)) =0$. 
We then have 
\begin{equation*}
\frac {\de \Phi_{\mu}} {\de |\mu^s|} (x_0) = \lim_{j \to \infty} \frac {\Phi_{\mu} (Q_{0} (x_0, r_j))}{|\mu^s|(Q_{0} (x_0,r_j))} = \lim_{j \to \infty}  \lim_{n \to \infty} \frac{ \displaystyle\int_{Q_{0} (x_0,r_j)} f\left( \frac {x} {\eps_n}, u_n (x) \right)\de x } {|\mu^s|(Q_{0} (x_0,r_j))}.
\end{equation*}
Define
$$w_{j,n}(y):= \frac {u_n(x_0+r_jy)} {t_j}$$ 
for $y \in Q_{0}$, where
$$t_j := \frac{|\mu^s|(Q_{0} (x_0,r_j))} {r_j^N} .$$ 
By changing variables we obtain
\begin{equation*}
\frac {\de \Phi_{\mu}} {\de |\mu^s|} (x_0) = \lim_{j \to \infty} \lim_{n \to \infty} \frac1 {t_j} \int_{Q_{0}}  f \left( \frac {r_j y} {\eps_n} + \frac {x_0} {\eps_n}, t_j w_{j,n}(y) \right)\,\de y.
\end{equation*}
Note that 
$$\frac {x_0} {\eps_n} = \left[ \frac {x_0} {\eps_n} \right] + \left< \frac {x_0} {\eps_n} \right>$$
and let $\gamma_n:=\langle x_0/\eps_n\rangle$. 
Thus, using the periodicity of $f$ with respect to the first variable, 
we have that 
\begin{equation*}
\frac {\de \Phi_{\mu}} {\de |\mu^s|} (x_0) = \lim_{j \to \infty} \lim_{n \to \infty}\frac1 {t_j} \int_{Q_{0}}  f \left( \frac {r_j y} {\eps_n} + \gamma_n, 
t_j w_{j,n}(y) \right)\,\de y.
\end{equation*}
We also have that
$$\lim_{j \to \infty} \lim_{n\to \infty} \int_{Q_{0}} w_{j,n} (y) \varphi(y) \,\de y = \langle \tau, \varphi \rangle$$
for every $\varphi \in C({\overline Q_{0}})$, where $\tau$ is a tangent measure to $\mu^s$ at $x_0$ (see Definition \ref{tan}).
Notice that we can choose the sequence of $r_j$ in such a way that all the properties above hold  and we also have $|\tau|(\partial Q_{0})=0$ (see Lemma \ref{lemmarin} and Remark \ref{suc-front}).
We also have
$$\lim_{j \to \infty} \lim_{n\to \infty} \int_{Q_{0}} |w_{j,n} (y)| \,\de y =  \frac {\de\Lambda} {\de |\mu^s|} (x_0).$$
Moreover, as ${\cal A}w_{j,n} \to 0$ in $W^{-1,q}$ for some $q \in \left(1, \frac{N}{N-1}\right)$, we can find a sequence ${\hat w}_j= w_{j,n_j}$ such that
$${\hat w}_j \weakst \tau, \qquad {\cal A}{\hat w}_j \to 0,$$
and 
\begin{equation}\label{369}
\frac {\de \Phi_{\mu}} {\de |\mu^s|} (x_0) = \lim_{j \to \infty}  \frac1 {t_j}\int_{Q_{0}}  f \left( \frac {r_j y} {\eps_{n_j}} + \gamma_{n_j}, t_j  {\hat w}_j (y) \right)\,\de y,
\end{equation}
where $\gamma_{n_j}$ is a subsequence of $\gamma_n$.

Note that ${\cal A} \tau=0$ and that we have
\begin{equation}\label{formamedtang}
\tau = \frac {\de \mu^s} {\de |\mu^s|} (x_0) |\tau|(x{\cdot}\omega_1,\ldots,x{\cdot}\omega_k),
\end{equation}
which means that $|\tau|$ is invariant for translations in the directions of the complement orthogonal of ${\cal V}_{x_0}$, that is, in the directions of $\mathrm{span} \{\omega_{k+1},\ldots,\omega_N\}$.
Indeed, we have
$${\cal A} \tau = \left( \sum_i A^{(i)} |\tau|_{x_i} \right) \frac {\de \mu^s} {\de |\mu^s|} (x_0)=0,$$
which implies that the vector $(|\tau|_{x_1},\ldots,|\tau|_{x_N})$ belongs to $\mathrm{span}\{\omega_1,\ldots,\omega_k\}$.

We regularize $\tau$ (see Remark \ref{suc-front}) and construct a sequence
$$v_j = \rho_{\eps_{n_j}} * \tau,$$
defined in $Q_{x_0}$. 
As $v_j$ has the form as in \eqref{formamedtang} we can extend it by $Q_{0}$-periodicity to a function ${\tilde v}_j$ defined on $\R{N}$, and we still have 
$${\cal A} {\tilde v}_j = 0$$
in $\R{N}$, that is, the jumps of $v_j$ are not penalized by ${\cal A}$ (see \cite{BCMS} for the details).

We now apply Lemma \ref{375} to the sequences $\{\hat w_j\}$ and $\{\tilde v_j\}$, with $f_n=f$ for all $n\in\N{}$, to obtain that there exists a sequence $\{z_i\}\subset L^q_{Q_{0}-\per}(\R{N};\R{d})$, $\int_{Q_{0}} z_j=0$, $\cA z_j=0$, $z_j\wsto0$ such that, from \eqref{376}, \eqref{369} becomes
\begin{equation*}
\frac {\de \Phi_{\mu}} {\de |\mu^s|} (x_0) \geq \lim_{j \to \infty} \frac1 {t_j}  \int_{Q_{0}} f \left( \frac {r_j y} {\eps_{n_j}} + \gamma_{n_j}, t_j  \hat u_j (y) \right)\,\de y
\end{equation*}
where
$$\hat u_j (y) = z_j(y) + {\tilde v}_j(y)$$
is $L^q_{Q_{0}-\per} (\R{N}; \R{d})$, ${\cal A}u_j=0$, $u_j \weakst \tau$.
%
%
%
%
%
%

The objective is to modify $\{\tilde u_j\}$ in order to obtain periodicity in the directions of the coordinate axes. Consider cut-off functions $\theta_i \subset C_c^{\infty} (Q_{0}; [0,1])$ such that $\theta_i (x) \equiv 1$ in $\left( 1- \frac1 {i} \right) Q_{0}$. Thus, using the Lipschitz continuity and the condition $\tau (\partial Q_{0})=0$, we have

\begin{equation*}
\frac {d \Phi_\mu} {d |\mu^s|} (x_0) \geq \liminf_{i\to\infty} \liminf_{j\to\infty}  \frac1 {t_j} \int_{Q_{0}} f \left(\frac {r_j y} {\eps_{n_j}} + \gamma_{n_j}, t_j \theta_i (y) \hat u_j(y) \right) \,dy.
\end{equation*}
After appropriate diagonalization we get a sequence $\{u_{j_i}\}$ such that
\begin{equation*}
\frac {d \Phi_\mu} {d |\mu^s|} (x_0) \geq \liminf_{i\to\infty}   \frac1 {t_{j_i}} \int_{Q_{0}} f \left(\frac {r_{j_i} y} {\eps_{n_{j_i}}} + \gamma_{n_{j_i}}, t_{j_i} \hat u_{j_i} (y) \right) \,dy,
\end{equation*}
where $\hat u_{j_i} \wsto \tau$, ${\cal A}\hat u_{j_i} \to 0$ in $W^{-1,q}$. 
By a change of variables we obtain
\begin{equation*}
\frac {d \Phi_\mu} {d |\mu^s|} (x_0) \geq \liminf_{i\to\infty}   \frac1 {t_{j_i}} \ave_{T_{j_i} Q_{0}} f \left(x+\gamma_{n_{j_i}}, t_{j_i} U_i \left( \frac {x} {T_{j_i}} \right) \right) \,dx,
\end{equation*}
where $T_{j_i}:= \frac {r_{j_i}}{\epsilon_{n_{j_i}}} \to \infty$ is a sequence of no necessarily integers and $U_i := \hat u_{j_i}$. 
Now choose a cube $T_iQ$ such that $T_{j_i} Q_0 \subset \subset T_iQ$ and extend $U_ i \left( \frac {x} {T_{j_i}} \right)$ by $T_i Q$-periodicity. We can choose, for instance, 
\be\label{360}
T_i = [C_N T_{j_i} ] +1,
\ee
where $C_N = \sqrt{N}$ is the diagonal of unit cube in $\R{N}$.
We then have
\begin{equation*}
\frac {d \Phi_\mu} {d |\mu^s|} (x_0) \geq \liminf_{i\to\infty}   \frac1 {t_{j_i}}  \frac1 {T_{j_i}^N} \left [\int_{T_i Q} f \left(x+\gamma_{n_{j_i}}, t_{j_i} U_i \left( \frac {x} {T_{j_i}} \right) \right) \,dx - \int_{T_i Q \setminus T_{j_i} Q_0} f(x+\gamma_{n_{j_i}},0) \,dx\right],
\end{equation*}
and, as the second integral vanishes as $i\to\infty$, by using \eqref{360} we obtain
\begin{equation*}
\frac {d \Phi_\mu} {d |\mu^s|} (x_0) \geq \liminf_{i\to\infty}   \frac{C_N^N}{t_{j_i}} \ave_{T_i Q} f \left(x+\gamma_{n_{j_i}}, t_{j_i} U_i \left( \frac {x} {T_{j_i}} \right) \right) \de x,
\end{equation*}         
and by changing variables again we get
\begin{equation}\label{361}
\frac {d \Phi_\mu} {d |\mu^s|} (x_0) \geq \liminf_{i\to\infty} \frac{C_N^N}{t_{j_i}} \int_{Q} f \left(T_i y+\gamma_{n_{j_i}}, t_{j_i} U_i \left( \frac {T_i y} {T_{j_i}} \right) \right) \de y.
\end{equation}   
As ${\cal A} U_i \left( \frac {T_i} {T_{j_i}}\cdot \right) \to 0$ and $U_i \left( \frac { T_i} {T_{j_i}} \cdot\right)$ is $Q$-periodic, $U_i \left( \frac {T_i} {T_{j_i}}\cdot \right)$ fulfills the hypotheses of Proposition \ref{projection}, and therefore \eqref{361} becomes
\begin{equation*}
\frac {d \Phi_\mu} {d |\mu^s|} (x_0) \geq \liminf_{i\to\infty} \frac{C_N^N}{t_{j_i}} \int_{Q} f \left(T_i y+\gamma_{n_{j_i}}, t_{j_i}  \left( \int_Q U_i \left( \frac {T_i x} {T_{j_i}} \right) \de x + V_i(y) \right) \right) \de y,
\end{equation*}  
where 
$$V_i(y) := \cP \left(U_i \left( \frac {T_i y} {T_{j_i}} \right) - \int_Q U_i \left( \frac {T_i x} {T_{j_i}} \right) \de x \right).$$
It is trivial to verify that $\int_Q V_i(y)\de y=0$. 
By changing variables back, we get
\begin{equation*}
\frac {d \Phi_\mu} {d |\mu^s|} (x_0) \geq \liminf_{i\to\infty}    \frac{C_N^N}{t_{j_i}} \ave_{T_i Q} f \left(x+\gamma_{n_{j_i}}, t_{j_i}  \left( \int_Q U_i \left( \frac {T_i x} {T_{j_i}} \right) \de x + V_i\left( \frac {x} {T_i} \right) \right) \right) \de x.
\end{equation*} 
Now, since $\int_Q U_i \left( \frac {T_i x} {T_{j_i}} \right) \de x \to \frac1{C_N^N} \frac {\de \mu^s} {\de |\mu^s|} (x_0)$, by using the Lipschitz condition we get
\begin{equation*}
\begin{split}
\frac {d \Phi_\mu} {d |\mu^s|} (x_0) \geq & \liminf_{i\to\infty}    \frac{C_N^N}{t_{j_i}}\ave_{T_i Q} f \left(x+\gamma_{n_{j_i}}, t_{j_i}  \left( \frac1{C_N^N} \frac {\de \mu^s} {\de |\mu^s|} (x_0) + V_i\left( \frac {x} {T_i} \right) \right) \right) \de x \\
& -\limsup_{i\to\infty} \frac{LC_N^N}{t_{j_i}}\ave_{T_i Q} \left|\int_Q U_i \left( \frac {T_i x} {T_{j_i}} \right) \de x - \frac1{C_N^N} \frac {\de \mu^s} {\de |\mu^s|} (x_0) \right|\de x,
\end{split}
\end{equation*}
and the $\limsup$ vanishes.
Therefore, by \eqref{222} 
we obtain
\be\label{365}
\begin{split}
\frac {\de \Phi_\mu} {\de |\mu^s|} (x_0) \geq & \liminf_{i\to\infty} \frac  {f_{{\cal A}-\hom}^{\gamma_{n_{j_i}}} \left(\frac{t_{j_i}}{C_N^N}\frac {\de \mu^s} {\de |\mu^s|} (x_0)\right)} {t_{j_i}/ C_N^N} \\
= & \liminf_{i\to\infty} \frac  {f_{{\cal A}-\hom} \left(\frac{t_{j_i}}{C_N^N}\frac {\de \mu^s} {\de |\mu^s|} (x_0)\right)} {t_{j_i}/ C_N^N} = {f^{\infty}_{{\cal A}-\hom}} \left(  \frac {\de \mu^s} {\de |\mu^s|} (x_0)\right).
\end{split}
\ee
Here we have used Proposition \ref{translations} and, in the last equality, the fact that $\cA$-quasiconvex functions are convex in the directions of the characteristic cone $\cC$ (see \cite[Proposition 3.4]{FM}). $\cA$-quasiconvexity for $f_{\cA-\hom}$ was proved in Remark \ref{fAhomAqc}, and this implies that the $\limsup$ in \eqref{104} is actually a limit.
This concludes the proof of \eqref{351} in the case $v_{x_0}\in\cC$.

\textbf{Case 2: $v_{x_0}\notin\cC$.}
This case is analogous to Case 1, but simpler, since the condition $v_{x_0}\notin\cC$ implies that $\cV_{v_{x_0}}=\{0\}$, and in turn that the tangent measure $\tau$ is given by 
\be\label{363}
\tau=v_{x_0}\cL^N.
\ee
In particular, we do not need to find a suitable rotated cube $Q_{0}$ to perform the homogenization and we do not need to regularize the tangent measure.

Since from Proposition \ref{Lipschitz} $f_{\cA-\hom}$ is Lipschitz continuous, by Lemma 4.2 in \cite{BCMS} there exists a sequence $\{t_j\}$ such that $t_j\to\infty$ as $j\to\infty$ and
\be\label{lim}
f_{\cA-\hom}^\infty(v_{x_0})=\lim_{j\to\infty}\frac{f_{\cA-\hom}(t_j v_{x_0})}{t_j}.
\ee
It is then possible to choose a sequence $\{\delta_j\}$ such that $\delta_j\to0$ as $j\to\infty$, $\Lambda(\partial Q(x_0,\delta_j)=0$, and
$$
t_j=\frac{|\mu^s|(Q(x_0,\delta_j))}{\delta_j^N}.
$$
With these sequences $\{t_j\}$ and $\{\delta_j\}$ we get the equivalent of equation \eqref{369} where the rotated cube $Q_{0}$ is replaced by the unit cube $Q$ and the sequence $\hat w_j$ converges weakly-* to the tangent measure $\tau$ in \eqref{363}. 
The proof proceeds now as in Case 1, with the integers $T_i$ in \eqref{360} now replaced by
$$
T_i=[T_{j_i}]+1
$$
(in particular, there is no need to introduce the constant $C_N$).
To conclude, we observe that the last equality in \eqref{365} now follows from \eqref{lim}.

\begin{remark} 
It easy to prove that under coercivity conditions any sequence of minimizers (or approximate minimizers in case the infimum is not attained) of the functional
$$I(u)= \int_{\Omega} f \left( \frac {x} {\epsilon_n}, u \right)\,dx$$
will converge (up to a subsequence) to the minimum points of the limit functional
$$\cF_{\cA-\hom}(\mu):=\int_\Omega f_{\cA-\hom}(u^a)\,\de x+\int_\Omega f_{\cA-\hom}^\infty\left(\frac{\de\mu^s}{\de|\mu^s|}\right)\,\de|\mu^s|.$$

In fact, in the particular case where the operator $\cal A$ admits an extension property, the condition ${\cal A} u_n \to 0$ in \eqref{202} can be replaced by ${\cal A} u_n = 0$, and the property above comes directly from known results in $\Gamma$-convergence on metric spaces.
\end{remark}

\noindent {\bf Acknowledgments.} 
Partial support for this research was provided by the Funda\c{c}\~ao para a Ci\^encia e a Tecnologia (Portuguese Foundation for Science and Technology) through the Carnegie Mellon Portugal Program under Grant FCT-UTA\_CMU/MAT/0005/2009 ``Thin Structures, Homogenization, and Multiphase Problems''. The authors warmly thank the Centro de An\'alise Matem\'atica, Geometria e Sistemas Din\^amicos (CAMGSD) at the Departamento de Matem\'atica of the Instituto Superior T\'ecnico, Universidade de Lisboa, where the research was carried out.

\end{document}